\documentclass[psamsfonts]{amsart}  
\usepackage{amssymb}
\usepackage[T1]{fontenc}
\usepackage[T1]{fontenc}
\usepackage[utf8]{inputenc}
\usepackage{bera}
\usepackage[portrait,margin=3cm]{geometry} 
\usepackage[mathscr]{euscript}
\usepackage[colorlinks=true,linkcolor=blue,citecolor=blue,urlcolor=blue]{hyperref} 
\usepackage{amsmath,amssymb,amsthm,amsfonts,amsbsy,latexsym,dsfont,color,graphicx,enumitem}
\usepackage[numeric,initials,nobysame]{amsrefs}

\usepackage{amssymb,bbm,ulem}

\usepackage{enumitem}
\usepackage{color,comment}

\newtheorem{thm}{Theorem}[section]
\newtheorem*{thm*}{Theorem}

\newtheorem{lemma}[thm]{Lemma}

\newtheorem{prop}[thm]{Proposition}

\theoremstyle{definition}

\newtheorem*{claim*}{Claim}

\newtheorem{example}{Example}[section]

\numberwithin{equation}{section}
\newtheorem*{acknowledgement}{Acknowledgement}
\newtheorem*{statements}{Statements and Declarations}

\begin{document}

\title[]{
A central limit theorem with explicit Lyapunov exponent and variance for products of $2\times2$ random non-invertible matrices}

\author[Benson]{Audrey Benson{$^\dag$}}
\thanks{\footnotemark {$\dag$} Research was supported in part by NSF Grant DMS-2316968.}
	\address{Department of Mathematics\\
		Union College\\
		Schenectady, NY 12308,  U.S.A.}
	\email{bensona2@union.edu }
	
\author[Gould]{Hunter Gould{$^\dag$}}
	\address{Department of Mathematics\\
		Union College\\
		Schenectady, NY 12308,  U.S.A.}
	\email{gouldh2@union.edu }	
	
\author[Mariano]{Phanuel Mariano{$^\dag$}}
	\address{Department of Mathematics\\
		Union College\\
		Schenectady, NY 12308,  U.S.A.}
	\email{marianop@union.edu}

\author[Newcombe]{Grace Newcombe{$^\dag$}}
	\address{Department of Mathematics\\
		Union College\\
		Schenectady, NY 12308,  U.S.A.}
	\email{newcombg@union.edu }

\author[Vaidman]{Joshua Vaidman{$^\dag$}}
	\address{Department of Mathematics\\
		Union College\\
		Schenectady, NY 12308,  U.S.A.}
	\email{vaidmanj@union.edu }



\begin{abstract}  

The theory of products of random matrices and  Lyapunov exponents have been widely studied and applied in the fields of biology, dynamical systems, economics,  engineering and statistical physics.   
We consider the product of an i.i.d. sequence of $2\times 2$ random non-invertible matrices with real entries. Given some mild moment assumptions we prove an explicit formula for the Lyapunov exponent and prove a central limit theorem with an explicit formula for the variance in terms of the entries of the matrices. We also give examples where exact values for the Lyapunov exponent and variance are computed. An important example where non-invertible matrices are essential is the random Hill's equation, which has numerous physical applications, including the astrophysical orbit problem.

\end{abstract}

\maketitle

\tableofcontents

\section{Introduction and Main Results}


Let $Y_{1},Y_{2},Y_{3},\dots$ be i.i.d. random matrices from the
ring $M_{2}\left(\mathbb{R}\right)$ of $2\times2$ matrices with
real entries. Let $S_{n}=Y_{n}\cdots Y_{2}Y_{1}$ to be their matrix
product and define the (top) \textbf{Lyapunov exponent} to be 
\begin{equation}
\lambda:=\lim_{n\to\infty}\frac{\mathbb{E}\left[\log\left\Vert S_{n}\right\Vert \right]}{n},\label{eq:Lyapunov-def}
\end{equation}
where $\left\Vert \cdot\right\Vert $ is any matrix norm. A theorem
of Fursternberg-Kesten \cite[Theorem 2]{FurKest} proves the analog of the Law of
Large Numbers for the product of random matrices by showing that 
\begin{equation}
\lim_{n\to\infty}\frac{\log\left\Vert S_{n}\right\Vert }{n}=\lambda\,\,\,a.s.\label{eq:LLN-FK}
\end{equation}
as long as $\mathbb{E}\log^{+}\left\Vert Y_{1}\right\Vert <\infty$
where $\log^{+}x:=\max\left\{ 0,\log x\right\} $. A
useful formula that can be used in obtaining exact values for $\lambda$ for matrices $Y_{i}$ with some distribution $\mu$ on a group $G\subset\text{SL}\left(d,\mathbb{R}\right)$
was found by Furstenberg \cite{Furstenberg1963}
\begin{equation}
\lambda=\int_{\text{SL}\left(d,\mathbb{R}\right)}\int_{\mathbb{P}^{d-1}}\log\left\Vert Ax\right\Vert d\nu\left(x\right)d\mu\left(A\right),\label{eq:Fursteberg-formula}
\end{equation}
where $\mathbb{P}^{d-1}$ is the projective space, and $\nu$
is a maximal $\mu$-invariant measure.

We also point to the classical works of  \cite{Tutubalin,Le_Page} (see also \cite{Bougerol,GL_CLT,Comtet}) where conditions
are given on the distribution of $\left\{ Y_{j}\right\} _{j\geq1}$
so that the following central limit theorem holds
\begin{equation}
\frac{1}{\sqrt{n}}\left(\log\left\Vert S_{n}\right\Vert -n\lambda\right)\overset{\mathcal{L}}{\to}N\left(0,\sigma^{2}\right)\label{eq:CLT-1}
\end{equation}
for some $\sigma^{2}\geq0$ where the convergence holds in law. 

There has been several works in obtaining explicit formulas for the Lyapunov
exponent given a random matrix distribution $\left\{ Y_{j}\right\} _{j\geq1}$
such as in the works of \cite{Bougerol,Crisanti,Forrester2013,Forrester-Zhang-2020,Kargin,Lima-Rahibe1994,Mannion,Marklof-etall2008,Newman1,Akemann-Burda-Kieburg2014}. But in many cases, one cannot obtain
explicit formulas for the Lyapunov exponent given a particular random matrix ensemble, even for $2\times 2$ random matrices. In fact, finding an exact $\mu-$invariant measure $\nu$ in \eqref{eq:CLT-1} is a difficult problem in itself \cite{Marklof-etall2008}. In situations where explicit expressions aren't available, one can approximate the Lyapunov exponent such
as in the works of \cite{Pollicott2010,Rajeshwari-etall-2020,Lemm-Sutter-2020,Protasov-Jungers-2013,Sturman-Thiffeault-2019,Viswanath2000,Jurga-Morris-2019}. Explicit formulas for the
variance $\sigma^{2}$ are also known though less studied, see \cite{Newman,Reddy-2019} for example. We point to the work of \cite[Proposition 2.1]{Newman}
 where explicit formulas are given for the Lyapunov exponent and variance given a condition on the distribution of the random matrices.  The condition given in \cite[Proposition 2.1]{Newman} is only valid for special cases and does not hold for general random matrices.

Most of the theory for products of random matrices require the matrices
to come from $\text{GL}\left(d,\mathbb{R}\right)$ and often require
conditions on the subgroup generated by the support of the distribution
for the given random matrix model. An example is the well-known formula for computing Lyapunov exponents by Furstenberg (\eqref{eq:Fursteberg-formula}, see also \cite[Theorem II.4.1]{Bougerol}).  If the matrices are non-invertible (also defined as singular), then
other than \eqref{eq:LLN-FK},  most of the results of the existing literature are not applicable. This holds especially true for central limit theorems with explicit variance for the products of non-invertible matrices. We do point to the works of \cite{Feng-Lo-Shen-2020,Freijo-Duarte-2022,Lima-Rahibe1994} for some results for non-invertible matrices. The main goal of this paper
is to find explicit formulas for the Lyapunov exponent $\lambda$
and prove a central limit theorem with an explicit variance formula
$\sigma^{2}$ in terms of the distribution of the matrix entries for the products of general $2\times 2$ non-invertible matrices. Our work fills this gap in the literature by finding explicit formulas, such as in \eqref{eq:Fursteberg-formula} for invertible matrices,  for both $\lambda$ and $\sigma^2$.

In particular, we consider the products of random matrices of the
form 
\[
Y_{j}=\left[\begin{array}{cc}
a_{j} & b_{j}\\
c_{j} & d_{j}
\end{array}\right],\,\,\,\,\,j\geq1
\]
where $Y_{j}$ are non-invertible matrices in $M_{2}\left(\mathbb{R}\right)$
and $\left\{ \left(a_{j},b_{j},c_{j},d_{j}\right)\right\} _{j\geq1}$
are i.i.d. sequences of random vectors. We will assume $a_j$ does not have an atom at zero. Since the $Y_{j}$ are non-invertible
then $a_{j}d_{j}-b_{j}c_{j}=0$ for $j\geq1$ and it 
suffices to consider $Y_{j}$ of the form 
\begin{equation}
Y_{j}=\left[\begin{array}{cc}
a_{j} & b_{j}\\
c_{j} & \frac{b_{j}c_{j}}{a_{j}}
\end{array}\right],\,\,\,\,\,j\geq1.\label{eq:matrix-model}
\end{equation}

As far as the authors know, the only paper to prove a CLT with an explicit variance for the
products of $2\times2$ random non-invertible matrices is the special case in \cite[Theorem 2.2]{Mariano-Panzo-2022}. In particular, in \cite{Mariano-Panzo-2022} the authors prove a CLT with an exact formula for $\sigma^2$ for  matrices of
the form 
\begin{equation}
Y_{j}=\left[\begin{array}{cc}
1 & x_{j}\\
\frac{1}{x_{j}} & 1
\end{array}\right],\label{eq:Hills-model}
\end{equation}
where $\left\{ x_{j}\right\} _{j\geq1}$ is an i.i.d. sequence of
random variables atomless at zero.  Unlike in the invertible matrix case, there is an interesting
and precise nondegeneracy condition that characterizes the entries $\left\{ x_{j}\right\} _{j\geq1}$  that give $\sigma^{2}=0$. These matrices are related to the
random Hill's equation studied in \cite{Adams-Bloch2008,Adams-Bloch2009,Adams-Bloch2010,Hills_2013}. In \cite{Hills}, Adams-Bloch-Lagarias were the first to prove an exact formula for the Lyapunov exponent. The Hill's equation has appeared in the literature with various applications such as in
 the modeling of lunar orbits \cite{Hill-1986} and various other astrophysical orbit problems \cite{Adams-Bloch2008}. See \cite{Hills_book} for various physical and engineering applications of the Hill's equation's.

We are now ready to state our main theorem. We prove an explicit formula
for the Lypunov exponent $\lambda$ and a central limit theorem with
an explicit formula for the variance $\sigma^{2}$ which is given in terms
of the entries $\left\{ \left(a_{j},b_{j},c_{j}\right)\right\} _{j\geq1}$
of $Y_{j}$. We point out that no assumption is given on the distribution
of the entries $\left\{ \left(a_{j},b_{j},c_{j}\right)\right\} _{j\geq1}$
other than a mild moment condition and that the entries are atomless at zero. Our results are the first to cover general $2\times 2$ non-invertible matrices which includes the special case given in \eqref{eq:Hills-model}. 
\begin{thm}
\label{thm:Main}Consider the random non-invertible matrices of the form
\[
Y_{j}=\left[\begin{array}{cc}
a_{j} & b_{j}\\
c_{j} & \frac{b_{j}c_{j}}{a_{j}}
\end{array}\right],\,\,\,j\geq1
\]
where $\left\{ \left(a_{j},b_{j},c_{j}\right)\right\} _{j\geq1}$
is an i.i.d. sequence of $\mathbb{R}\backslash\left\{ 0\right\} \times\mathbb{R}^{2}-$valued
random variables such that 
\begin{equation}
\mathbb{E}\left[\log^{+}\left(\left|a_{1}\right|+\left|c_{1}\right|\right)\right]<\infty\text{ and }\mathbb{E}\left[\log\left(1+\frac{\left|b_{1}\right|}{\left|a_{1}\right|}\right)\right]<\infty.\label{eq:MomentCondition1}
\end{equation}
Then $\left(\ref{eq:LLN-FK}\right)$ holds and the Lyapunov exponent
has the explicit expression
\begin{equation}
\lambda=\mathbb{E}\left[\log\left|a_{1}+\frac{b_{2}c_{1}}{a_{2}}\right|\right].\label{eq:Explicit-LyapuovFormula}
\end{equation}
Moreover, if we further assume that 
\begin{equation}
\mathbb{E}\left[\left(\log\left|a_{1}+\frac{b_{2}c_{1}}{a_{2}}\right|\right)^{2}\right]<\infty,\label{eq:MomentCondition2}
\end{equation}
then $\frac{1}{\sqrt{n}}\left(\log\left\Vert S_{n}\right\Vert -n\lambda\right)\overset{\mathcal{L}}{\to}N\left(0,\sigma^{2}\right)$
and the variance has the explicit expression 
\begin{equation}
\sigma^{2}=\mathbb{E}\left[\left(\log\left|a_{1}+\frac{b_{2}c_{1}}{a_{2}}\right|\right)^{2}\right]+2\mathbb{E}\left[\log\left|a_{1}+\frac{b_{2}c_{1}}{a_{2}}\right|\log\left|a_{2}+\frac{b_{3}c_{2}}{a_{3}}\right|\right]-3\lambda^{2}.\label{eq:Explicit-variance-formula}
\end{equation}
\end{thm}

The proof of Theorem \ref{thm:Main} relies on an explicit product formula for $S_n$, which will be given in Lemma \ref{lem:prod-formula}. With this product formula we can then compute $\lambda$ and use the theory of $m$-dependent sequences of random variables to prove the CLT using a theorem of Diananda in \cite[Theorem 2]{Diananda}. 

As demonstrated by the special case of the matrices related to the random Hill's equation
given in $\left(\ref{eq:Hills-model}\right)$, there
are non-trivial distributions $\left\{ \left(a_{j},b_{j},c_{j}\right)\right\} _{j\geq1}$
where the related CLT is degenerate with $\sigma^{2}=0$. In the case of the random Hill's matrices, a precise non-degeneracy
condition was given in \cite[Theorem 2.2]{Mariano-Panzo-2022} 
that completely characterizes the distribution
of the entries $\left\{ x_{j}\right\} _{j\geq1}$ that gives $\sigma^{2}=0$.
Unfortunately, such a precise characterization result does not seem tractable for
general non-invertible matrices. We use a result of Janson \cite{Janson-2023} (see also \cite{Janson}) to give
a non-degeneracy condition and prove some properties of the degenerate
case.

\begin{prop}
\label{prop:Characterization-sigma}Consider the setting of Theorem
\ref{thm:Main}. We have that $\sigma^{2}=0$ if and only if there
exists a function $\varphi:\mathbb{R}^{3}\to\mathbb{R}$ such that
for any $n\in\mathbb{N}$,

\[
\sum_{j=1}^{n}\left(\log\left[\left|a_{j}+\frac{b_{j+1}c_{j}}{a_{j+1}}\right|\right]-\lambda\right)=\varphi\left(\xi_{n+1}\right)-\varphi\left(\xi_{1}\right),\,\,\,\,a.s.
\]
where $\left\{ \xi_{j}\right\} _{j\geq1}=\left\{ \left(a_{j},b_{j},c_{j}\right)\right\} _{j\geq1}$.
Moreover, if $\sigma^{2}=0$ and the distribution of $\left(a_{j},b_{j},c_{j}\right)$
has an atom at $\left(a,b,c\right)$, then 
\[
\lambda=\log\left|a+\frac{bc}{a}\right|
\]
and for all $j\geq1$, 
\[
\left(a+\frac{b_{j}c}{a_{j}}\right)^{2}\left(a_{j}+\frac{bc_{j}}{a}\right)^{2}=\left(a+\frac{bc}{a}\right)^{4},\,\,\,\,a.s.
\]
\end{prop}

We organize the paper as follows. In Section \ref{Sec:Lyapunov}, we prove the product formula for $S_n$ and give the proof for the explicit formula of the Lyapunov exponent in Equation $\left(\ref{eq:Explicit-LyapuovFormula}\right)$. In Section \ref{SubSec:CLTproof} we give the proof of the CLT with explicit variance. We prove Proposition \ref{prop:Characterization-sigma} in Section \ref{SubSec:Degenerate}. Finally, we use Theorem \ref{thm:Main} to give examples with exact values of $\lambda$ and $\sigma^2$.


\section{Proof of Theorem \ref{thm:Main}: Explicit $\lambda$}\label{Sec:Lyapunov}

\subsection{Product Formula}

We first start by proving the following product formula. 
\begin{lemma}\label{lem:prod-formula}
The product of the matrices of the form 
\[
Y_{j}=\left[\begin{array}{cc}
a_{j} & b_{j}\\
c_{j} & \frac{b_{j}c_{j}}{a_{j}}
\end{array}\right],\,\,\,\,\,j\geq1
\]
can be expressed as 
\begin{equation}
S_{n}=\beta_{n}\left[\begin{array}{cc}
a_{n} & \frac{b_{1}}{a_{1}}a_{n}\\
c_{n} & \frac{b_{1}}{a_{1}}c_{n}
\end{array}\right]\label{eq:S_n_representation}
\end{equation}
where $\beta_{1}=1$ and 
\[
\beta_{n}=\prod_{j=1}^{n-1}\left(a_{j}+\frac{b_{j+1}c_{j}}{a_{j+1}}\right),n\geq1.
\]
\end{lemma}

\begin{proof}
The representation given in (\ref{eq:S_n_representation}) is clear
when $n=1$ since 
\[
S_{1}=\beta_{1}\left[\begin{array}{cc}
a_{1} & \frac{b_{1}}{a_{1}}a_{1}\\
c_{1} & \frac{b_{1}}{a_{1}}c_{1}
\end{array}\right]=\left[\begin{array}{cc}
a_{1} & b_{1}\\
c_{1} & \frac{b_{1}c_{1}}{a_{1}}
\end{array}\right].
\]
Suppose (\ref{eq:S_n_representation}) holds for $n$ , then we have
that 
\begin{align*}
 & S_{n+1}=Y_{n+1}S_{n}\\
 & =\beta_{n}\left[\begin{array}{cc}
a_{n+1} & b_{n+1}\\
c_{n+1} & \frac{b_{n+1}c_{n+1}}{a_{n+1}}
\end{array}\right]\left[\begin{array}{cc}
a_{n} & \frac{b_{1}}{a_{1}}a_{n}\\
c_{n} & \frac{b_{1}}{a_{1}}c_{n}
\end{array}\right]\\
 & =\beta_{n}\left[\begin{array}{cc}
a_{n}a_{n+1}+b_{n+1}c_{n} & a_{n+1}\frac{b_{1}}{a_{1}}a_{n}+b_{n+1}\frac{b_{1}}{a_{1}}c_{n}\\
c_{n+1}a_{n}+\frac{b_{n+1}c_{n+1}}{a_{n+1}}c_{n} & c_{n+1}\frac{b_{1}}{a_{1}}a_{n}+\frac{b_{n+1}c_{n+1}}{a_{n+1}}\frac{b_{1}}{a_{1}}c_{n}
\end{array}\right]\\
 & =\beta_{n}\left[\begin{array}{cc}
a_{n+1}\left(a_{n}+\frac{b_{n+1}c_{n}}{a_{n+1}}\right) & a_{n+1}\frac{b_{1}}{a_{1}}\left(a_{n}+\frac{b_{n+1}c_{n}}{a_{n+1}}\right)\\
c_{n+1}\left(a_{n}+\frac{b_{n+1}c_{n}}{a_{n+1}}\right) & \frac{b_{1}}{a_{1}}c_{n+1}\left(a_{n}+\frac{b_{n+1}c_{n}}{a_{n+1}}\right)
\end{array}\right]\\
 & =\prod_{j=1}^{n}\left(a_{j}+\frac{b_{j+1}c_{j}}{a_{j+1}}\right)\left[\begin{array}{cc}
a_{n+1} & \frac{b_{1}}{a_{1}}a_{n+1}\\
c_{n+1} & \frac{b_{1}}{a_{1}}c_{n+1}
\end{array}\right]=\beta_{n+1}\left[\begin{array}{cc}
a_{n+1} & \frac{b_{1}}{a_{1}}a_{n+1}\\
c_{n+1} & \frac{b_{1}}{a_{1}}c_{n+1}
\end{array}\right].
\end{align*}
\end{proof}

\subsection{Proof of exact $\lambda$ in Equation $\left(\ref{eq:Explicit-LyapuovFormula}\right)$}
\begin{proof}
We first take the $\log$ of the norm of the product representation of $S_n$ given in  (\ref{eq:S_n_representation}) of Lemma \ref{lem:prod-formula} 
to get 
\begin{equation}
\log\left\Vert S_{n}\right\Vert =\sum_{j=1}^{n-1}\log\left|a_{j}+\frac{b_{j+1}c_{j}}{a_{j+1}}\right|+\log\left\Vert \left[\begin{array}{cc}
a_{n} & \frac{b_{1}}{a_{1}}a_{n}\\
c_{n} & \frac{b_{1}}{a_{1}}c_{n}
\end{array}\right]\right\Vert .\label{eq:LogS_n}
\end{equation}
Consider the degenerate case where there is a positive probability
where $a_{j}+\frac{b_{j+1}c_{j}}{a_{j+1}}=0$. Note that in this case we have
$\log\left\Vert S_{n}\right\Vert =-\infty$ hence $\lambda=\mathbb{E}\left[\log\left|a_{1}+\frac{b_{2}c_{1}}{a_{2}}\right|\right]=-\infty$
as needed.

Let $\left\Vert \cdot\right\Vert $ denote the Hilbert-Schmidt norm
and compute
\begin{align*}
\left\Vert Y_{1}\right\Vert  & =\sqrt{a_{1}^{2}+b_{1}^{2}+c_{1}^{2}+\frac{b_{1}^{2}c_{1}^{2}}{a_{1}^{2}}}=\sqrt{a_{1}^{2}\left(1+\frac{c_{1}^{2}}{a_{1}^{2}}\right)\left(1+\frac{b_{1}^{2}}{a_{1}^{2}}\right)}.
\end{align*}
Using the moment condition $\left(\ref{eq:MomentCondition1}\right)$
with the property that $\log^{+}\left(xy\right)\leq\log^{+}\left(x\right)+\log^{+}\left(y\right)$
for any $x,y>0$ we have 
\begin{align*}
 & \mathbb{E}\left[\log^{+}\left\Vert Y_{1}\right\Vert \right]=\mathbb{E}\left[\log^{+}\sqrt{\left(a_{1}^{2}+c_{1}^{2}\right)\left(1+\frac{b_{1}^{2}}{a_{1}^{2}}\right)}\right]\\
 & \leq\mathbb{E}\left[\log^{+}\sqrt{\left(a_{1}^{2}+c_{1}^{2}\right)}+\mathbb{E}\left[\log^{+}\sqrt{\left(1+\frac{b_{1}^{2}}{a_{1}^{2}}\right)}\right]\right]\\
 & \leq\mathbb{E}\left[\log^{+}\left(\left|a_{1}\right|+\left|c_{1}\right|\right)\right]+\mathbb{E}\left[\log\left(1+\frac{\left|b_{1}\right|}{\left|a_{1}\right|}\right)\right]<\infty.
\end{align*}
Hence by a direct application of \cite[Theorem 2]{FurKest} we have that
$\lambda=\lim_{n\to\infty}\frac{\log\left\Vert S_{n}\right\Vert }{n}<\infty$
almost surely. 

We are left to compute the Lyapunov exponent $\lambda$ explicitly.
First we show some finite moment conditions. Note that for all
$n\geq1$, 
\begin{align*}
\mathbb{E}\left[\log\left\Vert \left[\begin{array}{cc}
a_{n} & \frac{b_{1}}{a_{1}}a_{n}\\
c_{n} & \frac{b_{1}}{a_{1}}c_{n}
\end{array}\right]\right\Vert \right] & =\mathbb{E}\left[\log\sqrt{a_{n}^{2}+\frac{b_{1}^{2}}{a_{1}^{2}}a_{n}^{2}+c_{n}^{2}+\frac{b_{1}^{2}}{a_{1}^{2}}c_{n}^{2}}\right]=\mathbb{E}\left[\log\sqrt{\left(a_{n}^{2}+c_{n}^{2}\right)\left(1+\frac{b_{1}^{2}}{a_{1}^{2}}\right)}\right]\\
 & \leq\mathbb{E}\left[\log^{+}\left(\left|a_{n}\right|+\left|c_{n}\right|\right)\right]+\mathbb{E}\left[\log\left(1+\frac{\left|b_{1}\right|}{\left|a_{1}\right|}\right)\right]<\infty,
\end{align*}
where we used $\left(\ref{eq:MomentCondition1}\right)$. Moreover,
we have that for all $j\geq1$, $-\infty\leq\mathbb{E}\log\left[\left|a_{j}+\frac{b_{j+1}c_{j}}{a_{j+1}}\right|\right]<\infty$
since

\begin{align*}
\mathbb{E}\left[\log\left[\left|a_{j}+\frac{b_{j+1}c_{j}}{a_{j+1}}\right|\right]\right] & \leq\mathbb{E}\left[\log\left[\left(\left|a_{j}\right|+\left|c_{j}\right|\right)\left(1+\frac{\left|b_{j+1}\right|}{\left|a_{j+1}\right|}\right)\right]\right]\\
 & \leq\mathbb{E}\left[\log^{+}\left[\left(\left|a_{j}\right|+\left|c_{j}\right|\right)\right]\right]+\mathbb{E}\left[\log\left[\left(1+\frac{\left|b_{j+1}\right|}{\left|a_{j+1}\right|}\right)\right]\right]<\infty.
\end{align*}

Taking the expected value of $\left(\ref{eq:LogS_n}\right)$ and using
the finite moment conditions shown above we have that
\begin{align*}
\lambda & =\lim_{n\to\infty}\frac{1}{n}\left(\mathbb{E}\left[\sum_{j=1}^{n-1}\log\left[\left|a_{j}+\frac{b_{j+1}c_{j}}{a_{j+1}}\right|\right]\right]+\mathbb{E}\left[\log\left[\left\Vert \left[\begin{array}{cc}
a_{n} & \frac{b_{1}}{a_{1}}a_{n}\\
c_{n} & \frac{b_{1}}{a_{1}}c_{n}
\end{array}\right]\right\Vert \right]\right]\right)\\
 & =\lim_{n\to\infty}\frac{1}{n}\left(\sum_{j=1}^{n-1}\mathbb{E}\left[\log\left[\left|a_{1}+\frac{b_{2}c_{1}}{a_{2}}\right|\right]\right]+\mathbb{E}\left[\log\left[\left\Vert \left[\begin{array}{cc}
a_{2} & \frac{b_{1}}{a_{1}}a_{2}\\
c_{2} & \frac{b_{1}}{a_{1}}c_{2}
\end{array}\right]\right\Vert \right]\right]\right)\\
 & =\mathbb{E}\left[\log\left|a_{1}+\frac{b_{2}c_{1}}{a_{2}}\right|\right],
\end{align*}
as needed. 
\end{proof}

\section{Proof of Theorem \ref{thm:Main}: the explicit CLT and non-degeneracy condition}\label{Sec:CLTproof}

\subsection{Proof of the CLT and Equation $\left(\ref{eq:Explicit-variance-formula}\right)$}\label{SubSec:CLTproof}
\begin{proof}
First note that by the condition $\left(\ref{eq:MomentCondition1}\right)$
combined with the $L^{2}$-moment condition in $\left(\ref{eq:MomentCondition2}\right)$ we have that
\[
\left|\lambda\right|=\left|\mathbb{E}\left[\log\left|a_{1}+\frac{b_{2}c_{1}}{a_{2}}\right|\right]\right|\leq\sqrt{\mathbb{E}\left[\left(\log\left|a_{1}+\frac{b_{2}c_{1}}{a_{2}}\right|\right)^{2}\right]}<\infty.
\] 

By $\left(\ref{eq:LogS_n}\right)$
we can write
\begin{align}
\log\left\Vert S_{n}\right\Vert -n\lambda & =\sum_{j=1}^{n-1}\left(\underset{A_{j}}{\underbrace{\log\left[\left|a_{j}+\frac{b_{j+1}c_{j}}{a_{j+1}}\right|\right]-\lambda}}\right)+\underset{B_{n}}{\underbrace{\log\left[\left\Vert \left[\begin{array}{cc}
a_{n} & \frac{b_{1}}{a_{1}}a_{n}\\
c_{n} & \frac{b_{1}}{a_{1}}c_{n}
\end{array}\right]\right\Vert \right]}}-\lambda.\label{eq:CLT-prod-formula}
\end{align}
Our goal will be to apply \cite[Theorem 2]{Diananda} to the sequence
$\left\{ A_{j}\right\} _{j\geq1}$ to obtain an explicit CLT. First,
it is clear that $\left\{ A_{j}\right\} _{j\geq1}$ is a $1-$dependent
sequence of random variables. One-dependence means that $\left\{ A_{1},\dots,A_{j}\right\} $
is independent of $\left\{ A_{k},A_{k+1},\dots\right\} $ whenever
$k-j>1$. Define $C_{j}$ by $C_{i-j}=\mathbb{E}\left[A_{i}A_{j}\right]$
for all $i,j\in\mathbb{Z}.$ We then have that $\mathbb{E}\left[A_{j}\right]=0$,
\[
C_{0}=\mathbb{E}\left[A_{j}^{2}\right]=\mathbb{E}\left[\left(\log\left|a_{1}+\frac{b_{2}c_{1}}{a_{2}}\right|\right)^{2}\right]-\lambda^{2},
\]
and 
\[
C_{1}=C_{-1}=\mathbb{E}\left[A_{2}A_{1}\right]=\mathbb{E}\left[\log\left|a_{1}+\frac{b_{2}c_{1}}{a_{2}}\right|\log\left|a_{2}+\frac{b_{3}c_{2}}{a_{3}}\right|\right]-\lambda^{2}.
\]
By 1-dependence it is clear that $C_{i}=0$ for all $\left|i\right|>1$. 

Applying Diananda's CLT for stationary $m$-dependent sequences from
 \cite[Theorem 2]{Diananda} we have 
\[
\frac{1}{\sqrt{n}}A_{j}\overset{\mathcal{L}}{\to}N\left(0,\sigma^{2}\right)
\]
where 
\[
\sigma^{2}=\sum_{j=-1}^{1}C_{j}=\mathbb{E}\left[\left(\log\left|a_{1}+\frac{b_{2}c_{1}}{a_{2}}\right|\right)^{2}\right]+2\mathbb{E}\left[\log\left|a_{1}+\frac{b_{2}c_{1}}{a_{2}}\right|\log\left|a_{2}+\frac{b_{3}c_{2}}{a_{3}}\right|\right]-3\lambda^{2}.
\]
By the moment condition $\left(\ref{eq:MomentCondition1}\right),\left(\ref{eq:MomentCondition2}\right)$
and Cauchy--Schwarz we have that $\sigma^{2}<\infty$. A standard
application of Markov's inequality gives that $\frac{1}{\sqrt{n}}B_{n}\to0$
in probability. Then, by Slutsky's theorem, we have 
\[
\frac{1}{\sqrt{n}}\left(\log\left\Vert S_{n}\right\Vert -n\lambda\right)\overset{\mathcal{L}}{\to}N\left(0,\sigma^{2}\right),
\]
as desired.
\end{proof}

\subsection{Proof of Proposition \ref{prop:Characterization-sigma}}\label{SubSec:Degenerate}

\begin{proof}
Suppose $\sigma^{2}=0$ in $\left(\ref{eq:Explicit-variance-formula}\right)$. We 
then apply a theorem of Janson in \cite[Theorem 8.1]{Janson-2023} to obtain a characterization
of the degenerate case. We follow the notation from \cite{Janson-2023}. Let $\left\{ \xi_{i}\right\} _{i\geq1}$
be the i.i.d. sequence of random variables given by $\left\{ \left(a_{i},b_{i},c_{i}\right)\right\} _{i\geq1}$
in $\mathcal{S}_{0}=\mathbb{R}\backslash\left\{ 0\right\} \times\mathbb{R}^{2}$
as in the setting of Theorem \ref{thm:Main}. Let $X_{i}:=A_{i}$
as in $\left(\ref{eq:CLT-prod-formula}\right)$. Recall that $X_{i}$
is a two-block factor of $\xi_{i}$ since $X_{i}=h\left(\xi_{i},\xi_{i+1}\right)$
where $h:\mathcal{S}_{0}^{2}\to\mathbb{R}$ is given by 
\[
h\left(\left(a_{i},b_{i},c_{i}\right),\left(a_{i+1},b_{i+1},c_{i+1}\right)\right)=\log\left[\left|a_{i}+\frac{b_{i+1}c_{i}}{a_{i+1}}\right|\right]-\lambda.
\]
Let $f:\mathbb{R}\to\mathbb{R}$ be the identity with $\mathcal{S}=\mathbb{R},\ell=1$
in the notation of \cite{Janson-2023} . The U-statistic $U_{n}$ in  \cite[Equation $(3.1)$]{Janson-2023}  is the usual sum
\[
S_{n}=\sum_{i=1}^{n}X_{j}.
\]
By \cite[Theorem 8.1(vi)]{Janson-2023} we have that $\sigma^2=0$ if and only if there exists a function $\varphi:\mathcal{S}_{0}\to\mathbb{R}$
such that 
\[
\sum_{j=1}^{n}X_{j}=\varphi\left(\xi_{n+1}\right)-\varphi\left(\xi_{1}\right)
\]
for all $n\geq1$ so that 
\begin{equation}
\sum_{j=1}^{n}\left(\log\left[\left|a_{j}+\frac{b_{j+1}c_{j}}{a_{j+1}}\right|\right]-\lambda\right)=\varphi\left(a_{j+1},b_{j+1},c_{j+1}\right)-\varphi\left(a_{1},b_{1},c_{1}\right)\label{eq:degenerate-1}
\end{equation}
for all $n\geq1$. 

\textbf{Atomic Case:} 

Now suppose $\xi_{j}$ has an atom, that is, there exists a $\left(a,b,c\right)\in\mathcal{S}_{0}$
so that $\mathbb{P}\left(\left(a_{j},b_{j},c_{j}\right)=\left(a,b,c\right)\right)>0$. Using $n=1$ in $\left(\ref{eq:degenerate-1}\right)$, there exists
a function $\varphi:\left(\mathbb{R}\backslash\left\{ 0\right\} \right)^{3}\to\mathbb{R}$
such that 
\[
\log\left[\left|a_{1}+\frac{b_{2}c_{1}}{a_{2}}\right|\right]-\lambda=\varphi\left(\left(a_{2},b_{2},c_{2}\right)\right)-\varphi\left(\left(a_{1},b_{1},c_{1}\right)\right),\,\,\,a.s.
\]
By independence we know $\mathbb{P}\left(\left(a_{1},b_{1},c_{1}\right)=\left(a,b,c\right),\left(a_{2},b_{2},c_{2}\right)=\left(a,b,c\right)\right)>0$
and
\begin{align*}
 & \mathbb{P}\left(\left(a_{1},b_{1},c_{1}\right)=\left(a,b,c\right),\left(a_{2},b_{2},c_{2}\right)=\left(a,b,c\right)\right)\\
 & =\mathbb{P}\left(\left(a_{1},b_{1},c_{1}\right)=\left(a,b,c\right),\left(a_{2},b_{2},c_{2}\right)=\left(a,b,c\right)\right)\mathbb{P}\left(\lambda=\log\left|a+\frac{bc}{a}\right|\right)
\end{align*}
so that $\mathbb{P}\left(\lambda=\log\left|a+\frac{bc}{a}\right|\right)=1$.
Hence 
\[
\lambda=\log\left|a+\frac{bc}{a}\right|.
\]

Now let $n=2$ in $\left(\ref{eq:degenerate-1}\right)$, so that by
rewriting we have
\[
\left|\left(a_{1}+\frac{b_{2}c_{1}}{a_{2}}\right)\left(a_{2}+\frac{b_{3}c_{2}}{a_{3}}\right)\right|=e^{2\lambda+\varphi\left(\left(a_{3},b_{3},c_{3}\right)\right)-\varphi\left(\left(a_{1},b_{1},c_{1}\right)\right)},\,\,\,\,a.s.
\]
Define $\alpha=\mathbb{P}\left(\left(a_{j},b_{j},c_{j}\right)=\left(a,b,c\right)\right)>0$
and use independence with the equation above to obtain
\begin{align*}
 & \alpha^{2}=\mathbb{P}\left(\left(a_{1},b_{1},c_{1}\right)=\left(a,b,c\right),\left(a_{3},b_{3},c_{3}\right)=\left(a,b,c\right)\right)\\
 & =\mathbb{P}\left(\left(a_{1},b_{1},c_{1}\right)=\left(a,b,c\right),\left(a_{3},b_{3},c_{3}\right)=\left(a,b,c\right),\left|\left(a_{1}+\frac{b_{2}c_{1}}{a_{2}}\right)\left(a_{2}+\frac{b_{3}c_{2}}{a_{3}}\right)\right|=e^{2\lambda+\varphi\left(\left(a_{3},b_{3},c_{3}\right)\right)-\varphi\left(\left(a_{1},b_{1},c_{1}\right)\right)}\right)\\
 & =\mathbb{P}\left(\left(a_{1},b_{1},c_{1}\right)=\left(a,b,c\right),\left(a_{3},b_{3},c_{3}\right)=\left(a,b,c\right)\right)\mathbb{P}\left(\left|\left(a+\frac{b_{2}c}{a_{2}}\right)\left(a_{2}+\frac{bc_{2}}{a}\right)\right|=\left|a+\frac{bc}{a}\right|^{2}\right)\\
 & =\alpha^{2}\mathbb{P}\left(\left|\left(a+\frac{b_{2}c}{a_{2}}\right)\left(a_{2}+\frac{bc_{2}}{a}\right)\right|=\left|a+\frac{bc}{a}\right|^{2}\right)
\end{align*}
so that
\[
\left|\left(a+\frac{b_{2}c}{a_{2}}\right)\left(a_{2}+\frac{bc_{2}}{a}\right)\right|=\left|a+\frac{bc}{a}\right|^{2},\,\,\,\,a.s.
\]
or
\[
\left(a+\frac{b_{2}c}{a_{2}}\right)^{2}\left(a_{2}+\frac{bc_{2}}{a}\right)^{2}=\left(a+\frac{bc}{a}\right)^{4},\,\,\,\,a.s.
\]
\end{proof}

\section{Examples: Exact Results}\label{Sec:Examples}

We use the results in Theorem \ref{thm:Main} to compute the exact
Lyapunov exponent $\lambda$ and variance $\sigma^{2}$ for the central
limit theorem for the products of the following random matrices.


\begin{example}[Binary]Let $Y_{j}=\left[\begin{array}{cc}
x_{j} & \frac{1}{x_{j}}\\
1 & \frac{1}{x_{j}^{2}}
\end{array}\right]$ where $x_{j}$ are independent and identically distributed random
variables taking values $a,b\neq0,-1$ and $\left(ab^{2}+1\right)\left(a^{2}b+1\right)\neq0$
with $\mathbb{P}\left(x_{j}=a\right)=p=1-\mathbb{P}\left(x_{j}=b\right)$. Then the Lyapunov exponent is given by 
\[
\lambda=p^{2}\log\left|a+\frac{1}{a^{2}}\right|+p\left(1-p\right)\log\left(\left|a+\frac{1}{b^{2}}\right|\left|b+\frac{1}{a^{2}}\right|\right)+\left(1-p\right)^{2}\log\left|b+\frac{1}{b^{2}}\right|
\]
and the variance is 
\begin{align*}
 & \sigma^{2}=p^{2}\left(1+\left(2-3p\right)p\right)\left(\log\left|a+\frac{1}{a^{2}}\right|\right)^{2}\\
 & -2\left(1-p\right)p^{2}\log\left|a+\frac{1}{a^{2}}\right|\left(\left(3p-1\right)\log\left(\left|a+\frac{1}{b^{2}}\right|\left|b+\frac{1}{a^{2}}\right|\right)+3\left(1-p\right)p\log\left(\left|b+\frac{1}{b^{2}}\right|\right)\right)\\
 & +\left(1-p\right)\left(1+3\left(p-1\right)p\right)p\left(\log\left(\left|a+\frac{1}{b^{2}}\right|\left|b+\frac{1}{a^{2}}\right|\right)\right)^{2}\\
 & +2\left(1-p\right)^{2}\left(3p-2\right)p\log\left(\left|a+\frac{1}{b^{2}}\right|\left|b+\frac{1}{a^{2}}\right|\right)\log\left|b+\frac{1}{b^{2}}\right|-\left(1-p\right)^{2}\left(3p-4\right)p\left(\log\left|b+\frac{1}{b^{2}}\right|\right)^{2}.
\end{align*}
\end{example}


In the next few examples we consider the non-invertible matrix 
\[
Y_{j}=\left[\begin{array}{cc}
x_{j} & x_{j}\\
y_{j} & y_{j}
\end{array}\right],
\]
which allow for a simplified computation of exact values. In particular,
using Theorem \ref{thm:Main} we have that 
\[
\lambda=\mathbb{E}\left[\left(\log\left|x_{1}+y_{1}\right|\right)\right].
\]
Note that by independence we have
\begin{align*}
\mathbb{E}\left[\log\left|a_{1}+\frac{b_{2}c_{1}}{a_{2}}\right|\log\left|a_{2}+\frac{b_{3}c_{2}}{a_{3}}\right|\right] & =\mathbb{E}\left[\log\left|x_{1}+y_{1}\right|\log\left|x_{2}+y_{2}\right|\right]=\lambda^{2}
\end{align*}
which allows for the simplification of the variance formula 
\[
\sigma^{2}=\mathbb{E}\left[\left(\log\left|x_{1}+y_{1}\right|\right)^{2}\right]-\lambda^{2}.
\]

\begin{example}[Uniform]
 Let $Y_{j}=\left[\begin{array}{cc}
x_{j} & x_{j}\\
y_{j} & y_{j}
\end{array}\right]$ where $x_{j},y_{i}$ are all independent and identically distributed
over the interval $\left[-a,b\right]$ where $-a\leq0<b$. Then the
Lyapunov exponent and variance is given by 
\[
\lambda=\begin{cases}
2\log2-\frac{3}{2}+\log b & ,a=0,b>0\\
2\log2-\frac{3}{2}+\log a & ,-a<0,b=0\\
\log\left(2b\right)-\frac{3}{2} & ,-a<0<b,a=b
\end{cases}\,\,\,\,\,\,\text{and }\,\,\,\,\,\,\sigma^{2}=\begin{cases}
\frac{5}{4}-2\left(\log2\right)^{2} & ,a=0,b>0\\
\frac{5}{4}-2\left(\log2\right)^{2} & ,-a<0,b=0\\
\frac{5}{4} & ,-a<0<b,a=b
\end{cases}.
\]
\end{example}


\begin{example}[Exponential] Let $Y_{j}=\left[\begin{array}{cc}
x_{j} & x_{j}\\
y_{j} & y_{j}
\end{array}\right]$ where $x_{j},y_{i}$ are all independent and identically distributed
as exponential random variables with parameter $\theta>0$. Then 
\[
\lambda=1-\gamma-\log\theta\,\,\,\,\,\,\,\,\,\text{and }\,\,\,\,\,\,\,\,\,\sigma^{2}=\frac{\pi^{2}}{6}-1,
\]
where $\gamma\approx.57721$ is the Euler-Mascheroni constant. 
\end{example}


\begin{example}[Cauchy] Let $Y_{j}=\left[\begin{array}{cc}
x_{j} & x_{j}\\
y_{j} & y_{j}
\end{array}\right]$ where $x_{j},y_{i}$ are all independent and identically distributed
as a standard Cauchy distribution. Then 
\[
\lambda=\log2\,\,\,\,\,\,\,\,\,\text{and }\,\,\,\,\,\,\,\,\,\sigma^{2}=\frac{\pi^{2}}{4}.
\]
\end{example}

\begin{acknowledgement}
This work was supported in part by a National Science Foundation  grant in the Division of Mathematical Sciences, Award \#2316968.
\end{acknowledgement}

\begin{statements}
The authors have no relevant financial or non-financial interests to disclose.
\end{statements}

\bibliographystyle{plain}	
\bibliography{matrix_bib}

\end{document}